%% file: main.tex
\documentclass[11pt]{amsart}

\input{header.tex}

\title{More Exotic $\RP^2$-knots and Homotopy Spheres}
\author{Judson Kuhrman}
\date{April 2025}

\begin{document}

\begin{abstract}
    We extend the infinite family of exotic embeddings $\mathbb{RP}^2 \hookrightarrow S^4$ constructed by Miyazawa to a strictly larger family of exotic embeddings, by showing that in place of the pretzel knot $P(-2, 3, 7)$, an infinite family of knots may be used as input to the construction. To this end, we prove that for any Montesinos knot of the form $K(2,3,|6s+1|)$, the branched double cover of the corresponding roll-spun knot is a homotopy sphere. This in turn produces a larger family of homotopy spheres and homotopy $\CP^2$s with potentially interesting involutions. We also observe that Miyazawa's homotopy sphere can be obtained from $S^4$ by a Gluck twist.
\end{abstract}
\maketitle
\vspace{-\baselineskip}

\section{Introduction}

The study of exotic knotted surfaces in $S^4$ begins with the construction by Finashin-Kreck-Viro \cite{finashin-kreck-viro} of infinitely many embeddings $( \mathbb{RP}^2 )^{\# 10} \hookrightarrow S^4$ all of which are topologically isotopic but smoothly mutually non-isotopic. These surfaces are smoothly distinguished by the diffeomorphism type of their respective branched double covers. Since then, further examples of exotic non-orientable knotted surfaces in $S^4$ with non-diffeomorphic branched double covers have been found, for example, by Finashin \cite{finashin}, Havens \cite{havens}, Levine-Lidman-Piccirillo \cite{levine-lidman-piccirillo}, and Mati\'c et. al. \cite{matic-et-al}. To date, there is no known exotic orientable knotted surface in $S^4$, despite strong theorems detecting topologically isotopic knotted surfaces. Notably, Conway-Powell \cite{conway-powell} prove that any 2-knot with knot group $\mathbb Z$ is topologically unknotted, but no exotic example has yet been found. 

An alternate approach to smoothly distinguishing knotted surfaces, introduced by Miyazawa \cite{miyazawa}, is to consider not just the total space of the respective double branched covers but the additional data of the branching involution. This approach, using invariants of smooth involutions on 4-manifolds, has garnered recent interest due to striking results. Notably, Miyazawa proves the existence of an infinite family of $\mathbb{RP}^2$-knots in $S^4$ that are topologically unknotted but smoothly mutually non-isotopic. In contrast, among the aforementioned examples, \cite{finashin} long held the record for smallest-known non-orientable genus at 6, with \cite{matic-et-al} finding examples of non-orientable genus 5.\footnote{These and other higher-genus examples should still be considered interesting in their own right, as they are all \emph{irreducible}: they cannot be formed by connect-sum with an unknotted surface.} Further, Miyazawa's construction yields interesting homotopy $S^4$s and homotopy $\mathbb{CP}^2$s with involutions that are not conjugate by a diffeomorphism to standard involutions. In the case of the homotopy $\mathbb{CP}^2$s, Hughes, Kim, and Miller \cite{hughes-kim-miller} show that the underlying smooth manifolds are all standard, and conclude the existence of smooth involutions on $\mathbb{CP}^2$ which are topologically but not smoothly conjugate to complex conjugation.

Let $P_\pm \subseteq S^4$ be an unknotted $\mathbb{RP}^2$-knot with normal euler number $\pm 2$. By an exotically unknotted $\mathbb{RP}^2$, we mean a smooth embedding $\mathbb{RP}^2\hookrightarrow S^4$ which is topologically but not smoothly isotopic to one of $P_\pm$. Miyazawa's examples are constructed by repeatedly connect-summing $P_\pm$ with copies of the roll-spin of the pretzel knot $P(-2,3,7)$. We will show that the same construction produces more exotically unknotted $\mathbb{RP}^2$s by replacing $P(-2,3,7)$ with the Montesinos knot $K(2,3,|6s+1|)$ for any $s\in \Z\setminus\{0,-1\}$. (Here, we mean the Montesinos knot whose branched double cover is tautologically $\Sigma(2,3,|6s+1|)$, cf. \cite{saveliev}.)

The invariant Miyazawa uses to distinguish surfaces comes from real Seiberg-Witten gauge theory. We will refer to this invariant of knotted spheres and $\RP^2$s as the real Seiberg-Witten invariant. The main input to Miyazawa's construction of exotic $\RP^2$-knots is a 2-knot $S$ such that $S\# P_\pm$ is topologically unknotted and such that the value of the real Seiberg-Witten invariant $|\deg(S)|$ is nontrivial. Miyazawa uses the fact that $S\#P_\pm$ is topologically unknotted to deduce that the branched double cover $\Sigma_2(S^4,S)$ is a homotopy sphere. We observe that this implication is bidirectional, simplifying the task of finding candidates for $S$. In \cite{kang-park-taniguchi}, the value of the real Seiberg-Witten invariant was computed for Montesinos knots of the form $K(2,3,|6s+1|)$ and shown to take all odd integer values. We will use the results of this computation to produce a strictly larger class of exotically unknotted $\RP^2$s. We obtain the following results.
\begin{theorem}\label{thm:spheres}
    For all $s\in\Z$, the branched double cover of the roll-spun knot $\rho K(2,3,|6s+1|)$ is a homotopy 4-sphere.
\end{theorem}
\begin{theorem}\label{thm:planes}
    There exist topologically unknotted embeddings of $\RP^2$ in $S^4$ for which Miyazawa's real Seiberg-Witten invariant takes all odd positive integer values.
\end{theorem}
\begin{theorem}\label{thm:involutions}
    There exist homotopy 4-spheres with involutions, and homotopy $\CP^2$s with involutions which are topologically conjugate to complex conjugation, for which Miyazawa's real Seiberg-Witten invariant takes all odd positive integer values. These involutions do not preserve any positive scalar curvature metric and are not smoothly conjugate to standard involutions on $S^4$ and $\CP^2$, respectively.
\end{theorem}

We observe that the branched double cover of a twist-roll-spun knot can be obtained as the result of surgery along a section of a 4-dimensional mapping torus. From this perspective, we are able to strengthen the theorem of Hughes-Kim-Miller \cite{hughes-kim-miller} that the branched double cover of a twist-roll-spun knot is unchanged after adding 4 twists, and prove that Miyazawa's homotopy spheres are (connect-sums of) Gluck twists on $S^4$.

\begin{theorem}\label{thm:gluck}
    For any knot $K\subseteq S^3$ and $m,n,k\in\Z$, The branched double covers of the $m$-twist $n$-roll spin of $K$ and the $(m+2k)$-twist $n$-roll spin of $K$ are related by a $k$-fold Gluck twist.
\end{theorem}

The rest of the paper is organized as follows. In Section \ref{section:branched-covers}, we recall the definition of twist-roll-spun knots and examine the structure of their branched double covers, and we prove Theorem \ref{thm:gluck}. In Section \ref{section:group-theory}, we consider the relationship between the fundamental group of the branched double cover of a knot, the fundamental groups of the branched double covers of its twist-roll-spins, and the knot groups of $\RP^2$-knots constructed by connect-sum. In Section \ref{section:brieskorn-spheres}, we exploit the close relationship between the fundamental groups of Brieskorn spheres and tilings of the hyperbolic plane to prove Theorem \ref{thm:spheres}, and obtain Theorems \ref{thm:planes} and \ref{thm:involutions} as consequences. Finally, in Section \ref{section:open-stuff}, we give some remaining open questions.\\

\noindent{\bf Acknowledgments.} The author thanks his advisor Ciprian Manolescu for many helpful insights and continual encouragement, and for partial support through the Simons Investigator grant. Thanks as well to Gary Guth, Seungwon Kim, and Maggie Miller for helpful conversations at various stages of this project. 

\section{Branched Double Covers of Twist-Roll-Spun Knots}\label{section:branched-covers}

\subsection{Deform-Spun 2-Knots}\label{section:deform-spun-definition} Litherland \cite{litherland} defines deform-spun 2-knots as follows. Let $K\subseteq S^3$ be a knot and let $\vartheta : S^3\xrightarrow{\sim} S^3$ be a diffeomorphism fixing a regular neighborhood $\nu(K)$ pointwise. We may form the mapping torus \begin{equation*}
    M_\vartheta = S^3 \tilde\times_\vartheta S^1.
\end{equation*}
Consider a point $x\in K$. Identify $S^3\setminus \nu(x)$ with $D^3$ and consider the knotted arc $K^\circ = K\setminus \nu(x)$. The boundary of $M_\vartheta\setminus (\nu(x)\tilde\times S^1)$ is canonically identified with $S^2\times S^1$. The deform-spun 2-knot $\vartheta K$ may be defined by a diffeomorphism of pairs
\begin{equation*}
    \big(S^4,\vartheta K\big) \cong \Big(M_\vartheta \setminus \big( \nu(x)\tilde\times S^1 \big),\, K^\circ\tilde\times S^1\Big)\cup_\partial \Big( S^2\times D^2, \partial K^\circ \times S^1 \Big).
\end{equation*}
The 2-knot $\vartheta K$ depends on $\vartheta$ only up to isotopy fixing $\nu(K)$ (pointwise). The group $\mathcal D(S^3,K)$ of \emph{deformations} of $K$ is the group of diffeomorphisms of $S^3$ fixing $\nu(K)$ modulo isotopy fixing $\nu(K)$. 
\begin{remark}
        Litherland actually defines deformations as diffeomorphisms of $D^3$ fixing $K^\circ\cup\partial D^3$, modulo \emph{pseudo}-isotopy fixing $K^\circ\cup \partial D^3$, or equivalently diffeomorphisms of $S^3$ fixing $K\cup\nu(x)$ modulo pseudo-isotopy fixing $K\cup\nu(x)$. Any such deformation has a representative fixing $\nu(K)$, and this representative is unique up to pseudo-isotopy fixing $\nu(K)$. The choice of regular neighborhood $\nu(K)$ is immaterial, since the canonical map to the group of deformations fixing a smaller regular neighborhood is an isomorphism. By the work of Waldhausen \cite{waldhausen} on Haken manifolds, the difference between isotopy and pseudo-isotopy is also immaterial.
\end{remark}
\begin{remark}
        Let $\mathrm{Emb}(K^\circ)$ be the space of proper embeddings $I\hookrightarrow D^3$ which are isotopic to $K^\circ$ and agree with $K^\circ$ on the boundary. Consider the fibration
        \begin{equation*}
        \mathrm{Diff}\big(D^3, K^\circ\cup \partial D^3\big) \longrightarrow \mathrm{Diff}\big(D^3, \partial D^3\big) \longrightarrow \mathrm{Emb}\big(K^\circ\big)
    \end{equation*}
    Hatcher \cite{hatcher} shows that the middle term $\mathrm{Diff}(D^3,\partial D^3)$ is contractible, so there is a homotopy equivalence \begin{equation*}
        \mathrm{Diff}\big(D^3, K^\circ\cup \partial D^3\big)\simeq \Omega\,\mathrm{Emb}\big(K^\circ\big).
    \end{equation*}
    Hence, deformations of $K$ may be viewed as homotopy classes of self-isotopies of $K^\circ$ fixing $\partial K^\circ$. 
\end{remark}

Let $T=\partial \nu(K^\circ) \cong T^2$ and parameterize a neighborhood of $T$ as $\nu(T) \cong T\times [0,1]\cong  \lambda\times\mu\times [0,1]$, where $\lambda,\mu$ are a longitude, meridian of $K$. Consider the diffeomorphism \begin{equation*}
    \tau^m\rho^n:S^3\xrightarrow{\sim} S^3
\end{equation*}
supported on $\nu(T)$ where \begin{equation*}
    \tau^m\rho^n(\theta_1,\theta_2,t) = (\theta_1 + 2\pi nt, \theta_2 + 2\pi mt, t), \quad (\theta_1,\theta_2,t)\in \nu(T).
\end{equation*}
The deform-spun 2-knot $\tau^m\rho^n K$ is the \emph{$m$-twist $n$-roll spin} of $K$.

\subsection{The Structure of the Branched Double Cover}\label{section:branched-cover-structure}
The branched double cover over a twist-roll-spun knot decomposes as \begin{equation}\label{eq:sphere-decomp}
    \Sigma_2(S^4,\tau^m\rho^n K) \cong \Big(  M_{\tilde\tau^m\tilde\rho^n} \setminus \big(\nu(\tilde x)\tilde \times S^1\big)\Big)\cup_\partial S^2\times D^2
\end{equation}
where $\tilde\tau,\tilde\rho$ are diffeomorphisms lifting $\tau,\rho$ to the branched double cover and $M_{\tilde\tau^m\tilde\rho^n}$ is the mapping torus. 

We proceed to describe this decomposition more explicitly. Consider the branched double cover $\Sigma_2(S^3,K)$. Let $\iota_K$ denote the branching involution, let $\tilde K\subseteq \Sigma_2(S^3, K)$ denote the fixed point set of $\iota_K$, and let $\tilde x$ denote the preimage of the point $x$ under the quotient map. Then, for an appropriately chosen 3-ball neighborhood $\nu(\tilde x)$, the involution $\iota_K$ restricts to an involution of $\Sigma_2(S^3,K)\setminus \nu(\tilde x)$ realizing the double cover of $D^3$ branched over $K^\circ$. Let $\iota_K^\circ$ denote the restricted branching involution and $\tilde K^\circ$ the fixed point set. 

Let $\tilde T = \partial\nu(\tilde K^\circ)$, let $\tilde \lambda\subseteq \tilde T$ denote a component of the preimage of $\lambda$, and let $\tilde \mu\subseteq \tilde T$ denote the preimage of $\mu$. We can then define diffeomorphisms $\tilde\tau,\tilde\rho: \Sigma_2(S^3, K) \xrightarrow{\sim}\Sigma_2(S^3, K)$ lifting $\tau,\rho$. View $\nu(\tilde T)$ as a collar of a slightly enlarged regular neighborhood $\nu(\tilde K)$. Both diffeomorphisms are the identity on $\nu(\tilde K)\setminus \nu(\tilde T)$. On $\nu(\tilde T)$, the diffeomorphism $\tilde \tau$ applies a half Dehn twist to every annulus of the form $\tilde \mu \times I$, while $\tilde \rho$ applies a full Dehn twist to every annulus of the form $\tilde \lambda \times I$. Outside of $\nu(\tilde K)$, the diffeomorphism $\tilde \tau$ is equal to the involution $\iota_K$, while $\tilde\rho$ is equal to the identity. The branched double cover of $\tau^m\rho^n K$ is \begin{equation}\label{eq:branched-cover-decomp}
    \Sigma_2\big(S^4,\tau^m\rho^n K\big) \cong \Big( \Sigma_2(D^3, K^\circ)\tilde\times_{\tilde\tau^m\tilde\rho^n} S^1\Big)\cup_\partial S^2\times D^2 .
\end{equation}

We will occasionally reference the fact that the branched double cover of $\tau^m\rho^n K$ is a rational homology sphere, so we record the homology here as a lemma. \begin{lemma}
    When $m$ is odd, $\Sigma_2(S^4,\tau^m\rho^n K)$ is an integer homology sphere. When $m$ is even, $\Sigma_2(S^4, \tau^m\rho^n K)$ is a rational homology sphere, with $H_1(\Sigma_2(S^4, \tau^m\rho^n K)) \cong H_1(\Sigma_2(S^3, K))$. In particular, when $m$ is even, $|H_1(\Sigma_2(S^4, \tau^m\rho^n K))| = \det(K)$. 
\end{lemma}
\begin{proof}
    By examining the Mayer-Vietoris sequence for the decomposition \eqref{eq:branched-cover-decomp}, we find that there is an isomorphism \begin{equation*}
        H_1\big(\Sigma_2\big(S^4,\tau^m\rho^n K\big);\Z\big) \cong H_1\big(\Sigma_2(S^3,K);\Z\big) / (1-\tilde\tau^m_*\tilde\rho^n_*).
    \end{equation*}
    When $m$ is even, $\tilde\tau^m_*\tilde\rho^n_*$ is the identity. When $m$ is odd, $\tilde\tau^m\tilde\rho^n$ is isotopic to the branching involution $\iota_K$, and $1-\tilde\tau^m_*\tilde\rho^n_* = 1 - \iota_{K*}$ is an isomorphism. Therefore, \begin{equation*}
        H_1\big(\Sigma_2\big(S^4,\tau^m\rho^n K\big);\Z\big) \cong \begin{dcases}
            0,& m\equiv1\pmod 2 \\
            H_1\big(\Sigma_2(S^3,K);\Z\big),& m\equiv 0\pmod 2
        \end{dcases}
    \end{equation*}
    In the even case, $H_1$ is finite of order $\det(K)$. The Euler characteristic of $\Sigma_2(S^4, \tau^m\rho^n K)$ is \begin{equation*}
        \chi\big(\Sigma_2\big(S^4, \tau^m\rho^n K\big)\big) = \chi\big( M_{\tilde\tau^m\tilde\rho^n} \setminus \big(\nu(\tilde x)\tilde \times S^1\big) \big) + \chi\big(S^2\times D^2\big) - \chi\big(T^3\big) = 0 + 2 - 0 = 2.
    \end{equation*}
    Since $\Sigma_2(S^4, \tau^m\rho^n K)$ has Euler characteristic $2$ and $H_1$ finite, Poincar\'e duality and universal coefficients imply that it is a rational homology sphere with $H_1=H_2$ and $H_3=0$. 
\end{proof}
\begin{corollary}\label{cor:homology-sphere}
    The branched double cover $\Sigma_2(S^4, \tau^m\rho^n K)$ is an integer homology sphere if and only if $m$ is odd or $\det(K)=1$.
\end{corollary}
\begin{corollary}\label{cor:homotopy-sphere}
    The branched double cover $\Sigma_2(S^4, \tau^m\rho^n K)$ is a homotopy sphere if and only if it is simply-connected.
\end{corollary}

We are interested in the question of when $\Sigma_2(S^4, \tau^m\rho^n K)$ is a homotopy sphere, and when it is diffeomorphic to the standard $S^4$. For odd-twist-spins the branched double cover is always diffeomorphic to $S^4$. 
\begin{lemma}
    If $m$ is odd, then $\Sigma_2(S^4, \tau^m K)$ is diffeomorphic to $S^4$.
\end{lemma}
\begin{proof}
    By Theorem 1.1 of \cite{hughes-kim-miller}, the branched double covers of $\tau^m K$ and $\tau^{m+4}K$ are diffeomorphic, so we may reduce to the case $m=\pm1$. Zeeman \cite{zeeman} proves that the $\pm1$-twist-spin of any knot is unknotted, so the branched double cover is $S^4$.
\end{proof}
\begin{remark}
    In general, when $(m,r)=1$, the $r$-fold cyclic branched cover of $\tau^m K$ is diffeomorphic to $S^4$, and the fixed knots in $S^4$ arising from these branched coverings are \emph{branched twist-spins}. Plotnick \cite{plotnick-2} proves that branched twist-spins are exactly the 2-knots which are fibered with finite-order monodromy.
\end{remark}

In other cases, the branched double cover is not generally simply-connected, so there is some inherent dependency on the input knot.

\subsection{Gluck Twists Along the Twin Sphere}\label{section:gluck-twist}

Hughes, Kim, and Miller \cite{hughes-kim-miller} study the branched double covers of twist-roll spun knots and show that for all knots $K$ and all parameters $m,n\in\Z$, there is a diffeomorphism \begin{equation*}
    \Sigma_2\big( S^4, \tau^m\rho^n K \big) \cong \Sigma_2(\tau^{m+4}\rho^n K).
\end{equation*}
That is, the smooth 4-manifold realizing the branched double cover of $\tau^m\rho^n K$ only depends on $m\pmod 4$. Their proof uses knot-complement surgery on unknotted tori in $S^4$. We will show that this fact can also be understood in terms of Gluck twists. This is the content of Theorem \ref{thm:gluck}.

Consider the decompositions \eqref{eq:sphere-decomp} and \eqref{eq:branched-cover-decomp} of $S^4$ and of the branched double cover. We will call the core of the $S^2\times D^2$ piece in each decomposition the \emph{twin sphere}. The union of a regular neighborhood of $\tau^m\rho^n K$ or of the fixed sphere upstairs and a regular neighborhood of the respective twin sphere together form what Montesinos \cite{montesinos-twin} calls a \emph{twin}, a pair of $S^2\times D^2$s plumbed at two points with opposite signs. As an abstract 4-manifold, a twin is diffeomorphic to the exterior of an unknotted torus in $S^4$. We now have the set-up to prove Theorem \ref{thm:gluck}.
\begin{proof}[Proof of Theorem \ref{thm:gluck}]
    We will prove that $\Sigma_2(S^4,\tau^m\rho^n K)$ and $\Sigma_2(S^4, \tau^{m+2k}\rho^n K)$ are related by a $k$-fold Gluck twist along the twin sphere. 
    
    First, assume that $m$ is even. As a diffeomorphism of $\Sigma_2(S^3, K)$, the lifted monodromy $\tilde\tau^m\tilde\rho^n$ is isotopic to the identity. In fact, $\Sigma_2(D^3, K^\circ)\tilde\times S^1 \cong (\Sigma_2(S^3, K) \times S^1) \setminus \nu(\gamma)$ where $\gamma$ is an embedded closed curve whose projection of $\gamma$ onto the $S^1$ factor is a degree 1 map of the circle. The diffeomorphism type of $\Sigma_2(D^3, K^\circ)\tilde\times S^1$ depends only on the (unbased, unoriented) homotopy class of $\gamma$, which in turn depends only on the homotopy class of its projection to $\Sigma_2(S^3, K)$. 
    
    An isotopy $\phi_t$ of $\Sigma_2(S^3, K)$ taking the identity to $\tilde\tau^m\tilde\rho^n$ can be described as follows. The isotopy $\phi_t$ is supported on a neighborhood of the fixed point set $\tilde K$. Equip $\nu(\tilde K) \cong \tilde K \times D^2$ with coordinates $(\theta_1, r, \theta_2)$ such that holding $\theta_2$ constant and varying $\theta_1$ traces out $\tilde\lambda$, and let $b(r)$ be a bump function with $b(r)=0$ for $r>1$ and $b(r)=1$ for $r<1/2$. Then, \begin{equation*}
        \phi_t(\theta_1,r,\theta_2) = (\theta_1 + 2\pi b(r) nt, r, \theta_2 + \pi b(r)mt).
    \end{equation*} 
    
    The curve $\gamma$ can be parameterized as $\gamma(t) = (\phi_t(\tilde x), t)$. The projection of $\gamma$ to $\Sigma_2(S^3, K)$ is homotopic to $n$ times the fixed point set of $\Sigma_2(S^3, K)$, independently of $m$. Both the mapping torus $\Sigma_2(D^3,K^\circ)\tilde\times S^1$ and $(\Sigma_2(S^3, K) \times S^1)\setminus \nu(\gamma)$ have their boundaries canonically identified with $S^2\times S^1$. Under the diffeomorphism defined by restricting the isotopy $\phi_t$, the respective identifications of the boundaries differ by an $m/2$-fold Gluck twist. Hence, the exteriors of the twin spheres in $\Sigma_2(S^4, \tau^m\rho^n K)$ and $\Sigma_2(S^4, \tau^{m+2k}\rho^n K)$ are diffeomorphic, but the framings of the boundaries differ by a $k$-fold Gluck twist.

    The proof in the case that $m$ is odd is essentially the same, except that $\tilde\tau^m\tilde\rho^n$ is isotopic to the branching involution, and the exterior of the twin sphere is the mapping torus $\Sigma_2(D^3, K^\circ)\tilde\times_{\iota_K} S^1$.
\end{proof}

Theorem \ref{thm:gluck} allows us to re-prove the theorem of Hughes-Kim-Miller on adding 4 twists, and to deduce that the branched double cover of $\rho P(-2,3,7)$ is a Gluck twist on $S^4$. Since Gluck twists stabilize after connect-summing with $\pm\CP^2$, we also obtain a re-proof that Miyazawa's homotopy $\CP^2$s are standard.

\begin{corollary}\label{cor:4-twists}(Theorem 1.1 in \cite{hughes-kim-miller})
    There is a diffeomorphism between the branched double covers $\Sigma_2(S^4,\tau^m\rho^n K)$ and $\Sigma_2(S^4,\tau^{m+4k}\rho^n K)$. If either of $\Sigma_2(S^4,\tau^m\rho^n K)$ or $\Sigma_2(S^4,\tau^{m+2k}\rho^n K)$ are simply-connected, then the two become diffeomorphic after one connect-sum with $\pm \CP^2$.
\end{corollary}
\begin{proof}
    Adding $4k$ twists induces an even-order Gluck twist on the branched double covers, which does nothing. Adding $2k$ twists either does nothing or induces a Gluck twist. If either branched cover is simply-connected then both are simply-connected, and connect-summing with $\pm\CP^2$ undoes any Gluck twisting (cf. Gompf-Stipsicz \cite{gompf-stipsicz}, Exercise 5.2.7).
\end{proof}
\begin{corollary}\label{cor:miyazawa-gluck-twist}(Corollary 1.2 in \cite{hughes-kim-miller})
    The branched double cover $\Sigma_2(\rho P(-2,3,7))$ is the result of a Gluck twist on $S^4$.
\end{corollary}
\begin{proof}
    Teragaito \cite{teragaito} observes that $\tau^{18}\rho P(-2,3,7)$ is unknotted, using the fact that 18-surgery on $P(-2,3,7)$ is a lens space. Therefore, $\Sigma_2(S^4,\rho P(-2,3,7))$ is the result of a Gluck twist on $\Sigma_2(S^4,\tau^{18}\rho P(-2,3,7))\cong S^4$.
\end{proof}
\begin{corollary}
    Let $P_\pm\subseteq S^4$ be an unknotted $\RP^2$ with normal Euler number $\pm 2$. For all $t\in\Z^{>0}$, the branched double cover $\Sigma_2(S^4, (\rho P(-2,3,7))^{\# t}\# P_\pm)$ is diffeomorphic to $\mp\CP^2$.
\end{corollary}
\begin{proof}
    Let $\Sigma = \Sigma_2(S^4,\rho P(-2,3,7))$. Then, \begin{equation*}
        \Sigma_2(S^4, (\rho P(-2,3,7))^{\# t}\# P_\pm) \cong \Sigma^{\#t}\#\big(\mp\CP^2\big).
    \end{equation*}
    When $t=1$, Corollary \ref{cor:miyazawa-gluck-twist} tells us that $\Sigma$ can be obtained from $S^4$ by a Gluck twist, and so $\Sigma \# (\mp\CP^2)\cong \mp\CP^2$. The claim for larger $t$ follows by induction.
\end{proof}

\section{Group-Theoretic Considerations}\label{section:group-theory}

\subsection{Weight 1 Groups}\label{section:weight-one}The \emph{weight} of a group $\pi$ is the minimal number of elements needed to normally generate $\pi$. Of particular interest are groups of weight 1. The knot group of any connected codimension 2 submanifold embedded in $S^n$ has weight 1, such a group being normally generated by a meridian. An element which normally generates $\pi$ is called a \emph{weight} element. 

We are interested in groups with trivial abelianization. In this case, an element is a weight element if and only if centralizing that element trivializes the group.
\begin{lemma}\label{lem:weight-vs-central}
    Let $\pi$ be a perfect group. Then, $\ft\in \pi$ is a weight element if and only if $\pi = [\pi, \ft]$.
\end{lemma}
\begin{proof}
    See \cite{476505}. If $\ft$ is a weight element, then $\pi / [\pi, \ft]$ is abelian. Therefore, $[\pi, \pi] \subseteq [\pi, \ft] \subseteq [\pi, \pi]$, so $[\pi,\ft] = [\pi,\pi] = \pi$.

    Conversely, if $\pi = [\pi, \ft]$ then $\pi$ is generated by commutators of the form $[\fs,\fr\ft\fr^{-1}]$ for $\fr,\fs\in \pi$. The subgroup generated by conjugates of $\ft$ contains the subgroup generated by commutators with conjugates of $\ft$, since \begin{equation*}
        \fs\fr\ft\fr^{-1}\fs^{-1}\fr\ft^{-1}\fr^{-1} = ((\fs\fr)\ft(\fr^{-1}\fs^{-1}))(\fr\ft^{-1}\fr^{-1}).
    \end{equation*}
\end{proof}
\begin{remark}
    If we forget the assumption that $\pi$ is perfect, the above argument shows that $[\pi,\ft]=[\pi,\pi]$ if and only if the normal subgroup generated by $\ft$ contains the commutator subgroup, i.e., $\pi/[\pi,\ft]$ is abelian if and only if $\pi/\langle\langle \ft\rangle\rangle$ is abelian.
\end{remark}

\subsection{The Fundamental Group of the Branched Double Cover}\label{section:fundamental-group} The main input to the construction of exotic $\RP^2$-knots in \cite{miyazawa} is a knot $K\subseteq S^3$ such that $\Sigma_2(S^4, \tau^m\rho^n K)$ is a homotopy sphere for some $m,n\in\Z$ satisfying $m+2n\equiv 2\pmod 4$. As discussed in Subsection \ref{section:gluck-twist}, when $m$ is even, $\Sigma_2(\tau^m\rho^n K)$ is the result of surgery on $\Sigma_2(S^3, K)\times S^1$ along a section whose projection to $\Sigma_2(S^3, K)$ is homotopic to $n$ times the fixed point set. Let $\gamma$ denote this section and $\ft\in \pi_1(\Sigma_2(S^3, K))$ the corresponding homotopy class, well-defined as an element of $\pi_1$ up to conjugacy. \begin{lemma}\label{lem:branched-cover-fundamental-group}
When $m$ is even, $\pi_1\big( \Sigma_2(S^4, \tau^m\rho^n K\big) )\cong \left\langle\,\pi_1( \Sigma_2(S^3, K)) \,\big|\, \ft\text{ central} \,\right\rangle$.
\end{lemma}
\begin{proof}
    The fundamental group of the mapping torus \begin{equation*}
        M_{\tilde\tau^m\tilde\rho^n} \cong \left(\Sigma_2\big(S^3, K\big) \times S^1\right) \setminus \nu(\gamma)
    \end{equation*}
    is $\pi_1(\Sigma_2(S^3, K)) \times \Z$, and gluing in $S^2\times D^2$ has the effect of identifying $\ft\in \pi_1(\Sigma_2(S^3, K))$ with the generator of $\Z$.
\end{proof}

In light of Corollary \ref{cor:homotopy-sphere} and Lemmas \ref{lem:weight-vs-central} and \ref{lem:branched-cover-fundamental-group}, we have the following. \begin{corollary}\label{cor:homotopy-sphere-weight}
    Let $m\in \Z$ be even. The branched double cover $\Sigma_2(S^4, \tau^m\rho^n K)$ is a homotopy 4-sphere if and only if $\ft\in\pi_1(\Sigma_2(S^4, \tau^m\rho^n K))$ is a weight element. In particular, $\Sigma_2(S^4, \tau^m\rho K)$ is a homotopy sphere if and only if the fixed knot in $\Sigma_2(S^3, K)$ normally generates the fundamental group.
\end{corollary}

In the torus knot case, Corollary \ref{cor:homotopy-sphere-weight} implies that the fixed set always normally generates.
\begin{proposition}\label{prop:torus-knot-normally-generates}
    Let $p,q\in\Z$ be odd and relatively prime. View the Brieskorn sphere $\Sigma(2,p,q)$ as the branched double cover of the torus knot $T(p,q)$. Then, the homotopy class of the fixed knot normally generates $\pi_1(\Sigma(2,p,q))$.
\end{proposition}
\begin{proof}
    By Corollary \ref{cor:homotopy-sphere-weight}, it suffices to prove that $\Sigma_2(S^4, \rho T(p,q))$ is simly-connected. Litherland \cite{litherland} proves $\rho T(p,q)\simeq \tau^{-pq} T(p,q)$. Hughes, Kim, and Miller \cite{hughes-kim-miller} prove that the branched double cover of $\tau^{-pq}T(p,q)$ is the same as that of $\tau^{-pq+4k}T(p,q)$ for any $k\in \Z$. Since $pq$ is odd, we may take $k$ such that $-pq + 4k = \pm 1$. Zeeman \cite{zeeman} proves that the $\pm1$-twist-spin of any knot is unknotted. Therefore, the branched double cover of $\rho T(p,q)$ is diffeomorphic to $S^4$.
\end{proof}

When $m$ is odd, a presentation for the fundamental group of the branched double cover can be derived from \cite{plotnick}.
\begin{lemma}
    When $m$ is odd, $\pi_1(\Sigma_2(\tau^m\rho^n K))\cong \langle\,\pi K\,|\,\mu = \lambda^{\pm n}, \mu^2\text{ central}\,\rangle$.
\end{lemma}
\begin{proof}
    Plotnick \cite{plotnick} considers 4-manifolds of the form \begin{equation*}
        \Sigma_A = \big(S^3\setminus\nu(K)\big)\times S^1\cup_A \big(S^4\setminus \nu(T)\big)
    \end{equation*}
    where $K\subseteq S^3$ is a knot and $T\subseteq S^4$ is an unknotted torus, $(S^3\setminus\nu(K))\times S^1$ and $S^4\setminus \nu(T)$ have their boundaries canonically identified with $T^3$, and $A\in\SL(3,\Z)$ is a matrix specifying a gluing of the boundaries. Section 7 of \cite{plotnick} gives a matrix presenting cyclic branched covers of $\tau^m\rho^n K$ with the order of branching relatively prime to $m$, and Section 2 of \cite{plotnick} gives moves on the matrix $A$ which leave $\Sigma_A$ unchanged. Combining these, we find that when $m$ is odd, the branched double cover of $\tau^m\rho^n K$ is $\Sigma_{A}$, where \begin{equation*}
        A = \begin{bmatrix}
            1 & 0 & 0 \\
            0 & 1 & 0 \\
            \mp n & 2n & 1
        \end{bmatrix}.
    \end{equation*}
    The sign on $\mp n$ depends on $m\pmod 4$. The fundamental group is then as claimed.
\end{proof}

\subsection{Connect-Sums with Projective Planes} Let $S\subseteq S^4$ be a 2-knot and let $P\subseteq S^4$ be an unknotted $\RP^2$. The knot groups $\pi S$ and $\pi (S\# P)$ are related by \begin{equation*}
    \pi (S\# P) = \left\langle\,\pi S\,\big|\,\mu^2 = 1\,\right\rangle 
\end{equation*}
where $\mu\in\pi S$ is a meridian. In constructing exotic $\RP^2$-unknots, Miyazawa \cite{miyazawa} observes that if $S\# P$ has knot group $\Z/ 2\Z$ then $\Sigma_2(S^4, S)$ is a homotopy $S^4$. Miyazawa argues by decomposing the branched double cover of $S\# P$, which is simply-connected, as a connect-sum of the branched double cover of $S$ with $\pm\CP^2$. In fact, the implication goes both ways. \begin{proposition}\label{prop:planes-vs-spheres}
    Let $S\subseteq S^4$ be a 2-knot and let $P\subseteq S^4$ be an unknotted $\RP^2$. The branched double cover $\Sigma_2(S^4, S)$ is a homotopy sphere if and only if $S\# P$ is topologically unknotted.
\end{proposition}
\begin{proof}
    Conway, Orson, and Powell \cite{conway-orson-powell} prove that $S\# P$ is topologically unknotted if and only if it has knot group $\pi(S\# P)\cong \Z/2\Z$. Our aim is therefore to prove that $\pi_1(\Sigma_2(S^4, S)) = 1$ if and only if $\pi\big(S\# P\big)\cong \Z/2\Z$.

    In fact, we claim that $\pi_1(\Sigma_2(S^4, S))$ is isomorphic to the commutator subgroup of $\pi(S\# P)$. The knot group $\pi S$ splits (non-canonically) as a semidirect product \begin{equation*}
        \pi S \cong \big(\pi S\big)' \rtimes \Z
    \end{equation*}
    where $(\pi S)'$ denotes the commutator subgroup and $\Z$ is generated by a meridian. The knot group $\pi(S\# P)$ is obtained from $\pi S$ by killing the square of a meridian, so we have an identification \begin{equation*}
        \pi(S\# P) \cong \frac{\big(\pi S\big)' \rtimes \Z}{2\Z}.
    \end{equation*}
    The fundamental group of the double cover of $S^4\setminus \nu S$ is $(\pi S)'\rtimes 2\Z$, the unique index 2 subgroup of $\pi S$. The branched double cover is obtained from the double cover by gluing in a sphere, killing all elements of the fundamental group represented by loops which are homotopic into the boundary. Such elements are exactly those that are conjugate to some element of the $2\Z$ subgroup. Therefore, we have an identification \begin{equation*}
        \pi_1\big(\Sigma_2\big( S^4, S \big)\big) \cong \frac{\big(\pi S\big)' \rtimes 2\Z}{2\Z}.
    \end{equation*}
    There is then a short exact sequence of groups \begin{equation*}
        1 \longrightarrow \pi_1\big(\Sigma_2\big( S^4, S \big)\big) \longrightarrow \pi(S\# P) \longrightarrow \Z/2\Z \longrightarrow 1.
    \end{equation*}
\end{proof}

\section{Construction of Homotopy Spheres}\label{section:brieskorn-spheres}

\subsection{Triangle Groups} Let $p,q,r\in\Z$ mutually relatively prime, such that $1/p + 1/q + 1/r < 1$. The triangle group $\Delta(p,q,r)\subseteq{\PSL(2,\R)}$ consists of the orientation preserving isometries of $\mathbb H^2$ which preserve a tiling $\cT$ by triangles with angles $\pi / p$, $\pi / q$, and $\pi / r$. Acting on the right, it has the presentation \begin{equation*}
    \Delta(p,q,r) = \presentation{a,b,c}{a^p=b^q=c^r=abc=1}
\end{equation*} 
where $a,b,c$ are counterclockwise rotations about the vertices of $\cT$. The preimage of $\Delta(p,q,r)$ in \scalebox{0.8}{${\SLtilde}$}$(2,\R)$ is the centrally extended triangle group $\Gamma(p,q,r)$. The fundamental group of the Brieskorn sphere $\Sigma(p,q,r)$ is isomoprhic to the commutator subgroup $\pi = \Gamma'\subseteq \Gamma$.

Either by considering a Dehn surgery description for $\Sigma(p,q,r)$ or by viewing $\Sigma(p,q,r)$ as an orbifold bundle over $S^2(p,q,r)$, it is straightforward to derive a presentation for $\pi$, \begin{equation*}
    \pi = \presentation{\fa,\fb,\fc,\fz}{\fz\text{ central},\fa^p=\fz^{-b_1}, \fb=\fz^{-b_2},\fc=\fz^{-b_3},\fa\fb\fc=\fz^{-e}}
\end{equation*}
where $b_1,b_2,b_3,e\in\Z$ satisfy \begin{equation*}
    b_1 qr + b_2 pr + b_3 pq = 1 + epqr.
\end{equation*}
The generator $\fz$ is represented by a regular fiber of $\Sigma$. Under the projection-induced map $\pi_1(\Sigma)\to\pi_1^{\operatorname{orb}}(S^2(p,q,r))\cong \Delta(p,q,r)$, the generators $\fa,\fb,\fc$ map to $a,b,c$, and the kernel of the projection-induced map is the center, generated by $\fz$.

\subsection{Branching Involutions}\label{section:branching-involutions}
Let $q,r\in\Z$ be odd and relatively prime, such that $1/p + 1/q + 1/r < 1$. The Brieskorn sphere $\Sigma(2,q,r)$ is simultaneously the double cover of $S^3$ branched over the torus knot $T(q,r)$ and over the Montesinos knot $K(2,q,r)$. Both branching involutions lift to fiber-preserving involutions of the universal cover \scalebox{0.8}{${\SLtilde}$}$(2,\R)$ and have a natural interpretation in terms of the action on $\cT$. The involution branching over $T(q,r)$ induces a rotation about an order 2 vertex of $\cT$, while the involution branching over $K(2,q,r)$ reverses the orientation of the fibers and induces a reflection about a line which is a union of edges of $\cT$.

\begin{lemma}\label{lem:words-representing-fixed-knots}
    Let $q,r\in\Z$ be odd and relatively prime such that $1/2+1/q+1/r<1$. Consider an element $\fm\in\pi_1(\Sigma(2,q,r))$ such that the image of $\fm$ in $\Delta(2,q,r)$ is \begin{equation}\label{eq:y-form-1}
        c^{\frac{r+1}2}b^{\frac{q+1}2}ab^{-\frac{q+1}2}c^{-\frac{r+1}2}a    
    \end{equation} 
    Up to multiplication by a power of $\fz$, $\fm$ represents the homotopy class of the fixed knot branching over $K(2,q,r)$.
\end{lemma}
\begin{figure}
    \includegraphics[width=\textwidth]{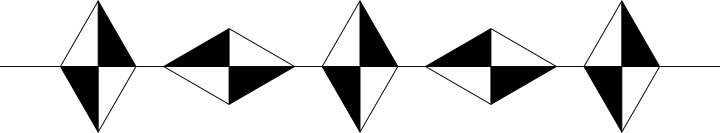}
    \caption{Order 2 vertices of $\cT$ arranged along the horizontal axis}\label{fig:alternating}
\end{figure}
\begin{figure}
    \includegraphics[width=0.5\textwidth]{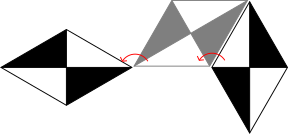}
    \caption{A sequence of rotations in $\Delta(2,q,r)$ taking one order 2 vertex to the next along the horizontal axis}\label{fig:rotations}
\end{figure}
\begin{proof}
    The branching involution on $\Sigma(2,q,r)$ induces a reflection of the plane about a line made of edges of $\cT$. An element of $\Delta(2,q,r)$ in the image of the homotopy class of the fixed knot is a hyperbolic transformation whose axis is such a line of reflection (cf. Remark 12.39 in \cite{burde-zieschang}). The fixed knot lies above an equator of $S^2(2,q,r)$ which connects the three orbifold points, and passes through each of the singular fibers of $\Sigma(2,q,r)$ twice. The translation length of a hyperbolic transformation represented by the fixed knot is $2(A+B+C)$, where $A$, $B$, and $C$ are the side lengths of a hyperbolic triangle with angles $\pi/2$, $\pi/q$, and $\pi/r$.

    The generators $a,b,c\in\Delta(2,q,r)$ represent rotations about the vertices of a triangle of orders $2$, $q$, and $r$, respectively. We can arrange that the order 2 vertex corresponding to the generator $a$ is centered at the origin in $\mathbb H^2$, and the vertices corresponding to $b$ and $c$ lie on the vertical and horizontal axes, respectively. Order 2 vertices of $\cT$ appear at regular intervals along the horizontal axis with alternating vertical and horizontal alignments, as shown in Figure \ref{fig:alternating}. Two consecutive horizontally aligned vertices are related by a hyperbolic transformation along the horizontal axis which preserves $\cT$. This transformation is in the conjugacy class represented by the homotopy class of the fixed knot.

    Figure \ref{fig:rotations} shows a sequence of rotations going from right to left taking a vertically aligned order 2 vertex to a horizontally aligned order 2 vertex. The first rotation is a rotation through angle $\pi - \pi / r$ about the vertex corresponding to $c$, hence this rotation is represented by $c^{(r-1)/2}$. The second rotation is conjugate to $b^{(q-1)/2}$. Specifically, with $\Delta(2,q,r)$ acting on the right, it is represented by $c^{-(r-1)/2}b^{(q-1)/2}c^{(r-1)/2}$. The composite is \begin{equation*}
        g = c^{\frac{r-1}2}c^{-\frac{r-1}2}b^{\frac{q-1}2}c^{\frac{r-1}2} = b^{\frac{q-1}2}c^{\frac{r-1}2}.
    \end{equation*}
    If we instead apply the transformation $aga=a^{-1}ga$, the vertically aligned vertex at the center is taken to the horizontally aligned vertex to its right. A hyperbolic transformation going from left to right and taking one horizontally aligned vertex to the next is then given by \begin{equation*}
        g^{-1}aga = c^{\frac{r+1}2}b^{\frac{q+1}2}ab^{-\frac{q+1}2}c^{-\frac{r+1}2}a.
    \end{equation*}
\end{proof}

\begin{proposition}\label{prop:montesinos-23q-normally-generates}
    Let $s\in\Z$ and view $\Sigma(2,3,|6s+1|)$ as the double cover of $S^3$ branched over a Montesinos knot. The conjugacy class of hyperbolic elements of $\Delta(2,3,|6s+1|)$ in the image of the homotopy class of the fixed knot generates $\Delta(2,3,|6s+1|)$. 
\end{proposition}
\begin{proof}
    In the cases $s=0$ and $s=-1$, the correponding Montesinos knot is a determinant 1 torus knot (specifically, the unknot and the torus knot $T(3,5)$, respectively), so we may appeal to Proposition \ref{prop:torus-knot-normally-generates}. Otherwise, we are in the hyperbolic case. Let $y\in\Delta=\Delta(2,3,r=|6s+1|)$ denote the element expressed in \eqref{eq:y-form-1}. We will prove that $\Delta / \langle \langle y\rangle\rangle = 1$. When $q=3$, we have $b^{(q+1)/2} = b^2 = b^{-1}$ and therefore \begin{equation*}
            y = c^{\frac{r+1}2}b^{-1}abc^{-\frac{r+1}2}a.
    \end{equation*}
    Using $b=(ca)^{-1}=ac^{-1}$, we then have \begin{equation*}
        y = c^{\frac{r+3}2}ac^{-\frac{r+3}2}a
    \end{equation*}
    and so in the quotient $y=1$, \begin{equation*}
        a = c^{\frac{r+3}2}ac^{-\frac{r+3}2}
    \end{equation*}
    so $a$ commutes with $c^{(r+3)/2}$ and hence also with $c^{r+3} = c^3$. Since $(3,r)=1$, $a$ must also commute with $c$. Any two of the generators $a,b,c$ together generate the whole group $\Delta$, and so $\Delta/\langle\langle y \rangle\rangle$ is abelian. Since $\Delta$ has trivial abelianization, $y$ must normally generate the whole of $\Delta$.
\end{proof}

We are now ready to prove Theorems \ref{thm:spheres}, \ref{thm:planes}, and \ref{thm:involutions}. \begin{proof}[Proof of Theorem \ref{thm:spheres}]
    The double cover of $S^3$ branched over $K=K(2,3,|6s+1|)$ is $\Sigma=\Sigma(2,3,|6s+1|)$, whose fundamental group is a central extension of $\Delta=\Delta(2,3,|6s+1|)$. Let $\fm\in\pi=\pi_1(\Sigma)$ be an element representing the homotopy class of the fixed knot branching over $K$. By Proposition \ref{prop:montesinos-23q-normally-generates}, the image of $\fm$ normally generates $\Delta$, so any element of $\pi$ may be written as $\fz^k\fg$ for some $k\in\Z$, where $\fz\in\pi$ is the generator of the center and $\fg$ belongs to the normal subgroup generated by $\fm$. Therefore, $\pi / \langle\langle \fm\rangle\rangle$ is cyclically generated by the image of $\fz$. Since $\pi$ has trivial abelianization, it follows that $\pi$ is normally generated by $\fm$. The theorem then follows from Corollary \ref{cor:homotopy-sphere-weight}.
\end{proof}
\begin{proof}[Proof of Theorem \ref{thm:planes} and Theorem \ref{thm:involutions}]
   Let $P$ denote a smoothly unknotted $\RP^2$-knot, and consider the connect-sum $P_s = (\rho K_s)\# P$ where $K_s$ is the Montesinos knot $K(2,3,|6s+1|)$ ($K_0$ is the unknot and $K_{-1}$ is the torus knot $T(3,5)$; otherwise $K_s$ is hyperbolic). By Theorem \ref{thm:spheres} and Proposition \ref{prop:planes-vs-spheres}, $P_s$ is topologically unknotted for all $s\in \Z$. Let $|\deg|$ be Miyazawa's real Seiberg-Witten invariant for 2-knots and $\RP^2$-knots, defined in \cite{miyazawa}. The real Seiberg-Witten invariants for the branched double covers of $\rho K_s$ and $P_s$ coincide. Further, in \cite{miyazawa}, it is shown that for all knots $K\subseteq S^3$, the value $|\deg(\rho K)| = |\deg((\rho K_s)\# P)|$ coincides with a corresponding 3-dimensional knot invariant $|\deg(K)|$. In \cite{kang-park-taniguchi}, the value of this 3-dimensional invariant evaluated on the Montesinos knot $K_s$ is shown to be \begin{equation*}
        |\deg(K_s)| =
            \begin{cases}
                \begin{aligned}
                    4j - 1&,\quad |6s+1| = 12j - 1\text{ or }12j - 5 \\
                    4j + 1&,\quad |6s+1| = 12j + 1\text{ or }12j + 5 
                \end{aligned}
            \end{cases}
   \end{equation*}
   and so in particular all positive odd values are achieved. The branched double cover of $(\rho K_s)\# P$ is a homotopy $\pm \CP^2$ and the branching involution is topologically conjugate to complex conjugation on standard $\pm\CP^2$. In \cite{miyazawa}, it is shown that the real Seiberg-Witten invariant is trivial for any involution which preserves a positive scalar curvature metric.
\end{proof}
\begin{remark}
    By taking connect sums of the roll-spun knots $\rho K(2,3,|6s+1|)$, we obtain many more 2-knots whose branched double covers are homotopy spheres and many more topologically unknotted $\RP^2$-knots. In particular, from any single knot $K(2,3,|6s+1|)$ we can construct an infinite family.
\end{remark}
\subsection{Torus Surgeries} In this section, we explain how twist-roll-spun knots whose branched covers are $S^4$ are implicitly related to nontrivial torus surgeries from $S^4$ to itself.

By a \emph{torus surgery}, we mean the operation of cutting out $T^2\times D^2$ from a 4-manifold and regluing it by a diffeomorphism of the boundary. We will call a torus surgery \emph{nontrivial} if the regluing diffeomorphism is not isotopic to the identity. Nontrivial torus surgeries may leave the 4-manifold unchanged. The branched double covers $\Sigma_2(S^4,\rho K(2,3,|6s+1|)$ can all be obtained from $S^4$ by a nontrivial torus surgery. More generally, this is true for $\Sigma_2(S^4,\tau^m\rho^n K)$ whenever $m$ is even and $\Sigma_2(S^3, K)$ is the result of Dehn surgery on a knot $K'\subseteq S^3$. 
\begin{proposition}\label{prop:torus-surgery}
    Suppose that $\Sigma_2(S^3, K)$ is the result of Dehn surgery on a knot $K'$. Then, the branched double cover $\Sigma_2(S^4, \tau^m\rho^n K)$ is the result of a nontrivial torus surgery on $S^4$.
\end{proposition}
\begin{proof}
    The branched double cover $\Sigma_2(S^4, \tau^m\rho^n K)$ can be obtained by gluing $S^2\times D^2$ to the mapping torus $(\Sigma_2(S^3, K)\times S^1) \setminus \nu(\gamma)$, where $\gamma$ is a closed embedded curve and $\nu(\gamma)$ is given the canonical framing of $\gamma$ if $m\equiv 0\pmod 4$ or the Gluck-twisted framing if $m\equiv 2\pmod 4$.  By assumption, $\Sigma_2(S^3, K)$ is the result of Dehn surgery on a knot $K'$. Let $g\in\pi K' = \pi_1(S^3\setminus \nu(K'))$ be represented by a curve homotopic to the projection of $\gamma$ to $\Sigma_2(S^3, K)$. 
    
    Let $x\in S^3\setminus\nu(K)$ and view $K'$ as contained in $D^3 = S^3\setminus \nu(x)\subseteq S^3$. Taking $x$ to be the basepoint, every element of $\pi K$ defines a point-pushing diffeomorphism of $D^3\setminus \nu(K)$ fixing the boundary, which extends to a diffeomorphism of $D^3$ fixing $\nu(K)$, well-defined up to isotopy fixing $\nu(K)$. Form the mapping torus $M_g = D^3\tilde\times S^1 \cong_{\text{rel. }\partial} D^3\times S^1$ whose monodromy is point-pushing along $g$. The knot $K'$ traces out a torus $K'\tilde\times S^1$ in $M_g$. By gluing $M_g$ to $S^2\times D^2$ along the boundary, we obtain two potentially different knotted tori in $S^4$, depending on whether we glue by the identity or by a Gluck twist. Let $T_0(K',g), T_1(K',g)\subseteq S^4$ denote the tori resulting from gluing by to the identity or by a Gluck twist, respectively.
    
    A nontrivial torus surgery along $K'\tilde \times S^1$ turns $M_g$ into $(\Sigma_2(S^3,K)\times S^1))\setminus \nu(\gamma)$. The branched double cover $\Sigma_2(S^4,\tau^m\rho^n K)$ is obtained by gluing in $S^2\times D^2$, with the framing depending on $m\pmod 4$. Commuting the operations of torus surgery and gluing in $S^2\times D^2$ proves the claim.
\end{proof}
\begin{remark}
    We may also allow for $m$ odd if we make the stronger assumption that $\Sigma_2(S^3, K)$ is Dehn surgery on a knot $K'$, such that the fixed knot comes from an unknot $U$ in $S^3$ disjoint from $K'$, and $K'$ is preserved by an involution of $S^3$ fixing $U$.
\end{remark}
\begin{remark}
    Banks \cite{banks} shows that the mapping class group of a punctured 3-manifold (in our case, a knot complement with an extra point removed) is an extension of the mapping class group of the original 3-manifold by the image of the point-pushing map, and the kernel of the point-pushing map is contained in the center of the fundamental group. Budney \cite{budney} computes the homotopy type of $\mathrm{Emb}(S^1,\R^3)$. The fundamental group of a given component can be shown to be an extension of the mapping class group of the punctured knot complement by $\Z/2\Z$, corresponding to turning.
\end{remark}
\begin{remark}
    Boyle \cite{boyle} defines spun and turned tori from knots in $S^3$. Analogously to the knotted tori defined in the proof of Proposition \ref{prop:torus-surgery}, we may define $T_{j}(K',g)$ for any $g\in \pi K'$. In our notation, $T_0(K',1)$ is the spun torus and $T_1(K',1)$ is the turned torus, where $1$ is the identity element.
\end{remark}

\begin{corollary}
    For all $s\in\Z$, the branched double cover $\Sigma_2(S^4,\tau^m\rho^n K(2,3,|6s+1|)$ can be obtained from $S^4$ by a nontrivial torus surgery.
\end{corollary}
\begin{proof}
    The branched double cover $\Sigma_2(K(2,3,|6s+1|))=\Sigma(2,3,|6s+1|)$ is Dehn surgery on the trefoil knot and on a twist knot. Now, apply Proposition \ref{prop:torus-surgery}.
\end{proof}

Now, consider the pretzel knot $P(-2,3,7)$. Teragaito \cite{teragaito} observes that $\tau^{18}\rho P(-2,3,7)$ is unknotted. The branched double cover $\Sigma_2(S^3,P(-2,3,7)) = \Sigma(2,3,7)$ is Dehn surgery on a trefoil knot $T(2,3)$. By Proposition \ref{prop:torus-surgery}, a nontrivial torus surgery along $T_1(T(2,3), g)$ produces $S^4$ as the result, where $g\in\pi T(2,3)$ is represented by a curve homotopic to the fixed knot in $\Sigma(2,3,7)$. The torus knot group $\pi T(2,3)$ has the presentation \begin{equation*}
    \pi T(2,3) = \left\langle\,a,b\,\big|\,a^2=b^3\,\right\rangle
\end{equation*} 
and, following the discussion in Subsection \ref{section:branching-involutions}, we may take $g$ to be \begin{equation*}
    g = \big(b^{-1}a\big)^{5}a\big(b^{-1}a\big)^{-5}a^{-1}.
\end{equation*}

Montesinos \cite{montesinos-twin} proves that any multiplicity 1 torus surgery on $S^4$ returns $S^4$ as the result. Larson \cite{larson} proves that any turned torus also admits nontrivial torus surgeries to $S^4$. The torus $T_1(T(2,3), g)$ does not fall in to either of these classes. 
\begin{proposition}
    The torus $T_1(T(2,3),g)$ is a torus admitting a nontrivial torus surgery to $S^4$ which is not topologically unknotted and is not a turned torus.
\end{proposition}
\begin{proof}
    We can distinguish $T_1(T(2,3),g)$ from all turned tori using the fundamental group. For any knot $K'$ and any $g'\in \pi K'$, the knot group $\pi T_j(K',g') = \pi_1(S^4 \setminus \nu(T_j(K',g')))$ has the presentation \begin{equation*}
        \pi T_j\big(K',g'\big) = \left\langle\,\pi K'\,\big|\,g'\text{ central}\,\right\rangle.
    \end{equation*}
    In particular, for the turned-or-spun torus $T_j(K',1)$, the knot group is $\pi K'$. On the other hand $\pi T_1(T(2,3), g)$ is a proper quotient of $\pi T(2,3)$. 
    
    From the discussion in Section 3.1 of \cite{boileau-et-al}, we can deduce that $\pi T(2,3)$ does not surject onto any knot groups other than itself and the group of the unknot. Further, no knot group is a proper quotient of itself (for torus knots, this can be deduced from the fact that torus knot groups are linear, using a theorem of Malcev \cite{malcev} that finitely generated linear groups are Hopfian). Therefore, $T_1(T(2,3),g)$ is not isotopic to $T_j(K',1)$ for any nontrivial knot $K'$. We can also rule out the case of the unknot, since $\pi T_j(T(2,3),g)$ surjects on to the group \begin{equation*}
        \left\langle\,a,b\,\big|\,a^2,b^3,\big(b^{-1}a\big)^{5}a\big(b^{-1}a\big)^{-5}a^{-1}\,\right\rangle
    \end{equation*}
    and Sage \cite{sagemath} can check that this group is not abelian.
\end{proof}

\section{Speculation and Open Questions}\label{section:open-stuff}

The fundamental question motivating this work is whether homotopy 4-spheres that arise as branched double covers of $S^4$ can be exotic, or whether they are necessarily standard. If any of the homotopy spheres of Theorem \ref{thm:spheres} were shown to be standard, this would imply the existence of an infinite family of inequivalent involutions on $S^4$ which do not preserve a positive scalar curvature metric. 
\begin{question}
    Are the homotopy spheres of Theorem \ref{thm:spheres} diffeomorphic to $S^4$?
\end{question}
Hughes, Kim, and Miller \cite{hughes-kim-miller} prove that the homotopy $\CP^2$s constructed by Miyazawa \cite{miyazawa} are standard. Their proof relies on the fact that an even surgery on the pretzel knot $P(-2,3,7)$ is a lens space. This proof does not immediately generalize to the other homotopy $\CP^2$s of Theorem \ref{thm:involutions}, the branched double covers of the $\RP^2$-knots $K(2,3,|6s+1|)\# P$, as $K(2,3,|6s+1|)$ does not generally admit a lens space surgery. Indeed, Ichihara and Jong \cite{ichara-jong} prove that $P(-2,3,7)$ is the \emph{only} hyperbolic Montesinos knot admitting a lens space surgery.
\begin{question}
    Are the homotopy $\CP^2$s of Theorem \ref{thm:involutions} diffeomorphic to $\CP^2$?
\end{question}
Similarly, our proof that $\Sigma_2(S^4, \rho P(-2,3,7))$ is a Gluck twist on $S^4$ uses the fact that $P(-2,3,7)$ has a lens space surgery, so we our method is insufficient to deduce the analogous result for other knots. 
\begin{question}
    Are the homotopy spheres of Theorem \ref{thm:spheres} Gluck twists on $S^4$?
\end{question}
Every possible value of the real Seiberg-Witten invariant $|\deg|$ is realized by $K(2,3,|6s+1|)$ for two different values of $s$, and so there are two topologically unknotted $\RP^2$-knots for which $|\deg|$ takes the same value. By taking connect-sums of roll-spun knots and using the multiplicative property of $|\deg|$, we can construct many more examples for which $|\deg|$ takes the same values. We are unable to determine whether the resulting $\RP^2$-knots are smoothly isotopic.
\begin{question}
    Are there 2-knots $S_1,S_2$ with $|\deg(S_1)|=|\deg(S_2)|$, such that $S_1\#P_\pm$ and $S_2\#P_\pm$ are both topologically unknotted but smoothly distinct?
\end{question}

Lastly, we remark that there are other examples of twist-roll spun knots which we can use to produce topologically unknotted $\RP^2$-knots, by analogous fundamental group computations. For example, Sage \cite{sagemath} can check that $(\rho K(2,5,7))\# P_\pm$ is topologically unknotted, but it is unclear whether this falls in to a larger pattern of Montesinos knots for which this construction works. There are many examples of 2-knots (e.g., \cite{kim}, \cite{longo}) such that connect-summing with $P_\pm$ produces a \emph{smoothly} unknotted $\RP^2$-knot.
\begin{question}
    For what input knots can we produce exotic $\RP^2$-knots by connect-sum with twist-roll-spun knots?
\end{question}

\end{document}

%% file: header.tex
\usepackage[sc]{mathpazo}

\usepackage{mathtools}
\usepackage{amssymb}
\usepackage{amsfonts}

\usepackage{graphicx}
\usepackage{hyperref}
\usepackage{tabularx}
\usepackage{multirow}
\usepackage{tikz-cd}
\usepackage{svg}
\usepackage{xcolor}
\usepackage{enumitem}

\theoremstyle{plain}
\newtheorem{theorem}{Theorem}
\newtheorem{proposition}{Proposition}[section]
\newtheorem{lemma}[proposition]{Lemma}
\newtheorem{corollary}[proposition]{Corollary}
\theoremstyle{definition}

\newtheorem{question}{Question}

\theoremstyle{remark}
\newtheorem{remark}[proposition]{Remark}

\newcommand{\Z}{\mathbb{Z}}

\newcommand{\R}{\mathbb{R}}

\newcommand{\RP}{\mathbb{RP}}
\newcommand{\CP}{\mathbb{CP}}

\newcommand{\SL}{\operatorname{SL}}

\newcommand{\PSL}{\operatorname{PSL}}
\newcommand{\SLtilde}{\widetilde{\SL}}

\newcommand{\presentation}[2]{\left\langle\,#1\,\big|\,#2\,\right\rangle}

\newcommand{\fa}{\mathfrak{a}}
\newcommand{\fb}{\mathfrak{b}}
\newcommand{\fc}{\mathfrak{c}}
\newcommand{\fg}{\mathfrak{g}}

\newcommand{\fz}{\mathfrak{z}}
\newcommand{\fm}{\mathfrak{m}}
\newcommand{\fr}{\mathfrak{r}}
\newcommand{\fs}{\mathfrak{s}}
\newcommand{\ft}{\mathfrak{t}}

\newcommand{\cT}{\mathcal{T}}